\documentclass{article}

\usepackage{geometry}
\usepackage{amsmath, amsthm, amssymb}
\usepackage[colorlinks=true,citecolor=blue]{hyperref}
\usepackage{lipsum}
\usepackage{authblk}
\usepackage{enumerate}
\usepackage{graphicx}
\usepackage{enumitem}

\usepackage{tikz}
\usepackage{amsmath}
\usetikzlibrary{calc}

\usetikzlibrary{calc,decorations.pathmorphing}

\tikzset{
    vertex/.style={
        circle,
        draw,
        thick,
        minimum size=2.4mm,
        inner sep=0pt,
        fill=white
    },
    edge/.style={thick},
    lab/.style={font=\itshape}
}

\usetikzlibrary{calc,decorations.pathmorphing}

\tikzset{
  edel/.style ={draw=black, densely dashed, line width=0.45pt},
}

\newtheorem{thm}{Theorem}
\newtheorem{lem}[thm]{Lemma}
\newtheorem{prop}[thm]{Proposition}
\newtheorem{defi}[thm]{Definition}
\newtheorem{cor}[thm]{Corollary}
\newtheorem{ex}[thm]{Example}
\newtheorem{obser}[thm]{Observation}
\newtheorem{rem}[thm]{Remark}
\newtheorem*{assum}{Assumption}

\title{Characterization of graphs $G$ where $G \in \mathrm{obs}^*(H)$ for some graph $H$}

\author[1]{Zahra Rahimi \thanks{zahra.paliz.rahimi@gmail.com}}
\author[1]{M. H. Shirdareh Haghighi \thanks{shirdareh@shirazu.ac.ir}}
\author[1]{Asma Namazi \thanks{asma.namazi@outlook.com}}

\affil[1]{Department of Mathematics and Computer Science, Shiraz University, Shiraz, Iran}

\begin{document}

\maketitle

\begin{abstract}
A full-homomorphism from a graph $G$ to a graph $H$ is a function on vertex sets that preserves adjacency and non-adjacency of vertices. A graph $G$ is called a minimal $H$-obstruction if it has no full-homomorphism to $H$ but every proper vertex induced subgraph of $G$ does. Such graphs can have at most $|V(H)|+1$ vertices. The set of minimal $H$-obstructions on $|V(H)|+1$ vertices is denoted by $\mathrm{obs}^*(H)$. In question 2 of the paper "Santiago Guzm{\'a}n-Pro, Full-homomorphisms to paths and cycles, Discrete Mathematics, 347(3):113800, 2024" it is asked if there is a characterization of those graphs $G$ that lie in $\mathrm{obs}^*(H)$ for some graph $H$. In this paper, we give a complete answer to this question.
\end{abstract}

\noindent \textbf{Keywords:} Full-homomorphism, minimal obstructions, realization of point determining graphs

\noindent \textbf{AMS subject classification 2020:} 05C60, 05C75




\section{Introduction}\label{intro}

In this paper we assume that our graphs are all simple and we use the notations of Bondy and Murty~\cite{bondy2008graph}. First we give the necessary definitions and a brief history of the subject. Let $G$ and $H$ be two graphs. A \textit{full-homomorphism} $\varphi:G \to H$ is a vertex mapping that preserves adjacency and non-adjacency of vertices. In other words, for every $x, y \in V(G)$, $x$ and $y$ are adjacent if and only if $\varphi(x)$ and $\varphi(y)$ are adjacent. 

Recall that the neighborhood of a vertex $x$ in a graph $G$, $N_G(x)$ is the set of vertices that are adjacent to $x$. The closed neighborhood of $x$, $N_G[x]$ is $N_G(x) \cup \{x\}$. The following definition is from~\cite{guzman2024full}.

\begin{defi}
In a graph $G$, two vertices $x$ and $y$ are called false twins if $N_G(x)=N_G(y)$, and true twins if $N_G[x]=N_G[y]$.
\end{defi}

So, the vertices $x$ and $y$ can be mapped to the same vertex if and only if they are false twins. A \textit{blow-up} of a vertex $x \in G$ is adding a new vertex $x'$ to $G$ such that $x'$ is adjacent to the neighborhood of $x$, therefore $x$ and $x'$ are false twins in the new graph. A \textit{blow-up} of a graph $G$ is obtained by blowing-up some (possibly none) vertices of $G$, a vertex can be blowed-up many times. Clearly, a graph $G$ has a full-homomorphism to $H$ if and only if it is a blow-up of some induced subgraph of $H$ (see~\cite{guzman2024full} for the definition of blow-up).

In~\cite{sumner1973point}, Sumner defines a \textit{point determining} graph to be a graph that for any $x \neq y$ in $G$ we have $N(x) \neq N(y)$. So, a point determining graph is a graph that has no false twins.

\begin{thm}\cite{sumner1973point}
Let $G$ be nontrivial and point determining, then there exists a vertex $x \in G$ such that $G-x$ is point determining, too.
\end{thm}

When $G$ is point determining, the set of vertices $x \in V(G)$ where $G-x$ is point determining is denoted by $G^\circ$. We call the vertices in $G^\circ$, \textit{nucleus vertices} of $G$.

For a given graph $H$, the graph $G$ is a \textit{minimal $H$-obstruction} if $G$ has no full-homomorphism to $H$, but every induced subgraph of $G$ does. The set of minimal $H$-obstructions is denoted by $\mathrm{obs} (H)$. Point determining graphs play an important role in studying minimal obstructions.

\begin{lem}~\cite[page 2]{guzman2024full}
If $G \in \mathrm{obs}(H)$, then $G$ is a point determining graph.
\end{lem}

Using the following theorem, we can define $\mathrm{obs}^*(H)$ for a graph $H$.

\begin{thm}~\cite{feder2008realizations}
If $H$ is a graph with $k$ vertices, then each minimal $H$-obstruction has at most $k+1$ vertices.
\end{thm}

The set of minimal $H$-obstructions with $|V(H)|+1$ is denoted by $\mathrm{obs}^*(H)$. We call a graph $G$ an \textit{obs*} graph if $G \in \mathrm{obs}^*(H)$ for some graph $H$.

A \textit{realization} of a point determining graph $H$ is a point determining graph $G$ such that for each $x \in G^\circ$, $G-x$ is isomorphic to $H$. We have the following lemma by Feder and Hell~\cite{feder2008realizations} which relates realizations to obs* graphs.

\begin{thm} \label{obs_is_realization}
Let $G$ and $H$ be graphs with $k+1$ and $k$ vertices, respectively. Then $G$ is a minimal $H$-obstruction if and only if it is a realization of $H$.
\end{thm}

In ~\cite[Question 2]{guzman2024full}, Guzm\'an-pro asks "For which graphs $G$ there is a graph $H$ such that $G$ is a minimal $H$-obstruction in $\mathrm{obs}^*(H)$?" In this paper, we answer this question.


\section{The structure of non-nucleus vertices} \label{non_nuc}

Let $G$ be a point determining graph and $x,y,z \in V(G)$. Based on~\cite{feder2008realizations} we would call $T=(x,\{y,z\})$ a triple of $G$ if $y$ and $z$ are false twins in $G-x$ (so one is adjacent to $x$ and the other is not, and neighbors of $y$ and $z$ are the same except in $x$). We call $x$ the middle vertex of this triple. Since every non-nucleus vertex creates a false twin, every such vertex is the middle vertex of some triple. Let us, as in~\cite{feder2008realizations}, define a relation $R$ between triples of $G$. Two triples are in relation $R$ if they share at least two vertices. This relation, although being reflexive and symmetric, is not necessarily transitive. Let $R'$ be the transitive closure of $R$. Then each equivalence class $C$ of $R'$ is called a \textit{triple component}.

A \textit{type-one} triple component consists of vertices $u_1, \ldots, u_r, v_1, \ldots, v_r$ with $r \geq 2$, and triples $(u_{i+1},\{v_i,v_{i+1}\})$ and $(v_i,\{u_i,u_{i+1}\})$, for all $1 \leq i \leq r-1$. The edges are either all $u_iv_j$  with $j<i$ or all $u_iv_j$ with $j \geq i$. We put them in a sequence $(u_1,v_1,u_2,v_2, \ldots, u_r, v_r)$ such that each three consecutive vertices $x,y,z$ are a triple $(y,\{x,z\})$. We call the vertices $u_1$ and $v_r$ the endpoints of this triple component; and the other vertices are called the middle vertices. We call a type-one triple component whose adjacencies are all $u_iv_j$ with $j \geq i$ an \textit{A-sequence}; and a type-one triple component whose adjacencies are all $u_iv_j$  with $j<i$ an \textit{N-sequence}. By a Bi-sequence, we mean an A-sequence or an N-sequence.

\noindent A \textit{type-two} triple component consists of vertices $u_1, \ldots u_r, u_{r+1}, v_1, \ldots, v_r$ with $r \geq 1$, and triples $(u_{i+1},\{v_i,v_{i+1}\})$ for all $1 \leq i \leq r-1$ and $(v_i,\{u_i,u_{i+1}\})$, for all $1 \leq i \leq r$. The edges are just all $u_iv_j$  with $j<i$. We can again put them in a sequence $(u_1,v_1,u_2,v_2, \ldots, u_r, v_r, u_{r+1})$ such that each three consecutive vertices  $x,y,z$ are a triple $(y,\{x,z\})$. In this triple component, the vertices $u_1$ and $u_{r+1}$ are endpoints. We call this sequence an \textit{odd sequence}.

\noindent A \textit{type-three} triple component consists of vertices $u_1, \ldots u_r, u_{r+1}, v_1, \ldots, v_r$ with $r \geq 1$, and triples $(u_{i+1},\{v_i,v_{i+1}\})$ for all $1 \leq i \leq r-1$ and $(v_i,\{u_i,u_{i+1}\})$, for all $1 \leq i \leq r$; together with the additional triple $(u_{r+1},\{v_r,u_1\})$. The edges must be all $u_iv_j$  with $j<i$. We put them in a sequence $(u_1,v_1,u_2,v_2, \ldots, u_r, v_r, u_{r+1}, u_1)$ such that each three consecutive vertices $x,y,z$ are a triple $(y,\{x,z\})$, including $v_r, u_{r+1}, u_1$. In this type of triple component, the endpoints coincide at $u_1$. We call this type of triple component a \textit{belt}.

\begin{rem} \label{induced}
The induced subgraphs of triple components are bipartite graphs with parts $V=\{v_1,\ldots,v_r\}$ and its complement. Also the induced subgraphs of an A-sequence and an N-sequence of the same order are bipartite complements of each other. The induced subgraphs of an odd sequence and a belt are the same but in an odd sequence, $v_r, u_{r+1}, u_1$ do not form a triple in the whole graph.
\end{rem}

\begin{thm}~\cite{feder2008realizations} \label{triple_comps}
Every triple component of a point determining graph must be a type-one, type-two, or type-three triple component.
\end{thm}


We say a triple component $C$ is on the vertex $x$ if $x$ is an endpoint of $C$.

In a sequence, two vertices $u_i$ and $u_j$ or $v_i$ and $v_j$ are called \textit{almost false twins}. In a belt, every two vertices are called almost false twins. Two almost false twin vertices have the same neighborhoods in $G$ outside of their own triple component.

Note that we may have triple components of order three. They are odd sequences and belts with just three vertices. These components are called \textit{small} triple components; small sequence and small belt, respectively.

\begin{lem}~\cite{sumner1973point} \label{edame}
If $G$ is a point determining graph and $T=(x,\{y,z\})$ is a triple such that $N(y)=N(z)-\{x\}$, then
\begin{enumerate}
\item If $y \notin G^\circ$, then there exists $w \in G-\{y,x\}$ such that $N(x)=N(w)-\{y\}$.
\item If $z \notin G^\circ$, then there exists $w \in G-\{z,x\}$ such that $N(w)=N(x)-\{z\}$.
\end{enumerate}
\end{lem}

\begin{cor}
In a point determining graph, every endpoint of a triple component is a nucleus vertex.
\end{cor}

\begin{proof}
Let $C$ be a triple component in a point determining graph $G$ and $a,b,c$ be the first three vertices of $C$. By Lemma~\ref{edame} if $a \notin G^\circ$ we have a triple $T_a=(a,\{w,b\})$ in $G$, for some $w \notin \{a,b\}$. The triples $T_a$ and $T$ share two vertices, hence they belong to the same triple component $C$. But $T_a$ is not in $C$, because there is no triple in $C$ with $a$ as its middle vertex. We do the same if $a$ is the last vertex of $C$.
\end{proof}

\begin{lem}~\cite{feder2008realizations} \label{no_odd}
If $G$ is a realization of some graph $H$, then $G$ cannot have a type-two triple component.
\end{lem}

So, we will only deal with A-sequences, N-sequences and belts in obs* graphs.

\begin{lem} \label{no_c_intersection}
If a non-nucleus vertex belongs to different triple components, then both triple components are small sequences.
\end{lem}

\begin{proof}
By~\cite[page 1642]{feder2008realizations}, if a non-nucleus vertex belongs to different triple components, each of these triple components are either a small sequence or a small belt. 
We will show that a small belt and any triple (whether in a small belt or small sequence) have no common non-nucleus vertices.

Suppose a small belt $C=(x_1,x_2,x_3,x_1)$ and a triple $T=(y_2,\{y_1,y_3\})$ intersect at the middle, that is $y_2 \in \{x_2,x_3\}$. Without loss of generality, suppose $x_2=y_2$ (The reversal of a belt is a belt, too). Note that $C$ and $T$ have no more vertices in common, because $C$ is a triple component. 
 
Since $T$ is a triple, the neighborhoods of $y_1$ and $y_3$ differ exactly in $y_2=x_2$. But since $C$ is a small belt, $x_2$ and $x_3$ are true twins and therefore have the same closed neighborhoods. Thus $x_3$ is also a difference in the neighborhoods of $y_1$ and $y_3$, which is a contradiction.
\end{proof}

Hence we have the following corollary.

\begin{cor}\label{no_intersection_in_obs}
A non-nucleus vertex in an obs* graph, cannot lie in two different triple components.
\end{cor}

Since the endpoints of each triple component are nucleus vertices, it is necessary to study the structure of $G[G^\circ]$. An obs* graph $G$ is, roughly speaking, a graph $K=G[G^\circ]$ with some appending Bi-sequences and belts as follows. Some A-sequences between two adjacent nucleus vertices, some N-sequences between two non-adjacent vertices, and some belts on some nucleus vertices.

In the following two sections we demonstrate how a suitable nucleus can be extended to an obs* graph by adding Bi-sequences and belts.

\begin{rem} \label{other_adj}
Once we have the structure of $K=G[G^\circ]$ in an obs* graph $G$ and we know what Bi-sequences or belts are on the vertices of $K$, we can determine all adjacencies and non-adjacencies in $G$. We do not have any odd sequences in an obs* graph, so every vertex is either itself a nucleus vertex or is almost false twin with a nucleus vertex. In a type-one triple component $(u_1,v_1,u_2,v_2,\ldots,u_r,v_r)$, the vertex $u_i$, $i>1$, is almost false twin with $u_1$, and the vertex $v_j$, $j<r$ is almost false twin with $v_r$ (the endpoints $u_1$ and $v_r$ are nucleus vertices). In a belt $(u_1,v_1,\ldots,u_r,v_r,u_{r+1},u_1)$, every non-nucleus vertex $z \neq u_1$ is almost false twin with the endpoint $u_1$ (which is nucleus vertex).
Let $u$ and $v$ be two vertices that are not in the same triple component of $G$. If $u$ is nucleus, set $u_0=u$ and if $u$ is non-nucleus, set $u_0$ to be the nucleus vertex that is almost false twin with $u$. We do the same for $v$ to get $v_0$. Now $u_0v_0$ is an edge, if and only if $uv$ is an edge.
Also note that any two triple components can only intersect in their endpoints, by Corollary~\ref{no_intersection_in_obs}.
\end{rem}


\section{The Structure of $G[G^\circ]$ in an obs* Graph $G$}

In this section, we study the structure of $G[G^\circ]$ in an obs* graph $G$. The main results of this section are Theorems~\ref{K_pd} and~\ref{the_big_thm} which state that for an obs* graph $G$, $G[G^\circ]$ is point determining and vertex-transitive. First observe that the degree of two nucleus vertices $x$ and $a$ are equal in $G$, since $G-x$ is isomorphic to $G-a$ (because $G$ is in fact the realization of $G-x$ for every nucleus vertex $x$ in $G$ by Theorem~\ref{obs_is_realization}).

\begin{thm} \label{K_pd}
If $G$ is an obs* graph, then $G[G^\circ]$ is point determining.
\end{thm}

\begin{proof}
Let $G$ be an obs* graph and suppose, to the contrary, that $K=G[G^\circ]$ is not point determining. Then $K$ has a pair $u,v$ of false twins.
We show that they are also false twins in $G$.

Let $c$ be a non-nucleus vertex of $G$.
Then there exists some nucleus vertex $a$ which is almost false twin with $c$ in a triple component $C$. Since $u$ and $v$ are false twins in $K$, $a$ is adjacent to $u$ if and only if it is adjacent to $v$. There is no triple component of $G$ with more than two nucleus vertices, thus at most one of $u$ or $v$ can be in $C$.

First suppose that $a \notin \{u,v\}$. Therefore, if none of $u$ and $v$ are in $C$, $c$ is also adjacent to $u$ if and only if it is adjacent to $v$. Now suppose that one of $u$ or $v$, say $u$, is in $C$. Then $C$ is a Bi-sequence and $a$ and $u$ are its two endpoints. In this case, by the adjacencies of a Bi-sequence, $c$ is adjacent to $u$ if and only if $a$ is adjacent to $u$. Furthermore, since $v$ is not in $C$, $c$ is adjacent to $v$ if and only if $a$ is adjacent to $v$. Thus $c$ is adjacent to $u$ if and only if it is adjacent to $v$.

Second suppose that $a \in \{u,v\}$, say $a=u$. Then if $C$ is a belt, then $a=u$ is not adjacent to any vertex of the belt, including $c$. If $C$ is a Bi-sequence, the vertices $a=u$ and $c$ are not adjacent (because they are almost false twins and belong to the same part of the bipartition of the Bi-sequence $C$). But the vertices $a=u$ and $v$ are false twins and therefore not adjacent. Thus, if $v \in C$, $C$ is an N-sequence with endpoints $u_1=u$ and $v_r=v$ and $c$ is a middle vertex $u_i$ for some $i$; hence $c$ and $v$ are non-adjacent. If $v \notin C$, then since the vertices $a=u$ and $c$ are almost false twins and $u$ is not adjacent to $v$, then $c$ is not adjacent to $v$.

So, in all cases $c$ is adjacent to $u$ if and only if it is adjacent to $v$. The proof is now complete.

\end{proof}

\begin{obser}\label{deg_seq}
{\normalfont
Let us have an observation about the degrees of the vertices in an obs* graph $G$. As we said earlier, the nucleus vertices have the same degree in $G$, say, $k$. Every non-nucleus vertex is involved in a Bi-sequence or belt. Consider an A-sequence $(u_1,v_1,u_2,v_2, \ldots, u_r,v_r)$. Since $u_1$ is a nucleus vertex, $\deg_G(u_1)=k$. Now $(v_1,\{u_1,u_2\})$ is a triple where $u_1$ is adjacent to $v_1$, thus we have $N(u_2)=N(u_1)-\{v_1\}$ and therefore $\deg_G(u_2)=k-1$. With the same argument we see that $N(u_3)=N(u_2)-\{v_2\}$ and thus $\deg_G(u_3)=k-2$. By repeating the arguments we see that $\deg_G(u_i)=\deg_G(v_{r-i+1})=k-(i-1)$, $1 \leq i \leq r$. So, the degrees of the vertices in an A-sequence would be $(k,k-(r-1),k-1,k-(r-2),k-2, \ldots, k-1, k-(r-1),k)$.

If we use a similar argument in an N-sequence $(u_1,v_1,u_2,v_2, \ldots, u_r,v_r)$, we get its degree sequence as $(k,k+(r-1),k+1,k+(r-2),k+2, \ldots, k+1,k+(r-1),k)$.

For the belt $(u_1,v_1,u_2,v_2, \ldots, u_r, v_r, u_{r+1}, u_1)$, the degrees would be $\deg_G(u_i)=\deg_G(v_{r-i+2})=k+(i-1)$, $1 \leq i \leq r+1$. Therefore the degrees of the vertices in this belt would be $(k,k+(r-1),k+1,k+(r-2),k+2, \ldots, k+1,k+(r-1),k)$.}
\end{obser}

\begin{cor}\label{deg_in_obs}
In an obs* graph, all nucleus vertices have the same degrees, say $k$; and for any non-nucleus vertex $u$, $\deg_G(u) \neq k$.
\end{cor}

\begin{thm}\label{regular_gnot}
If $G$ is an obs* graph, then $K=G[G^\circ]$ is regular.
\end{thm}

\begin{proof}
By Corollary~\ref{deg_in_obs}, all nucleus vertices have the same degree in $G$, say $k$. Let $a,b \in G^\circ$. We have $G-a \cong G-b$. It follows that the number of neighbors of $a$ and $b$ in the set of vertices of degree $k$ are the same. But by Corollary~\ref{deg_in_obs} only the nucleus vertices have degree $k$, hence $\deg_K(x)=\deg_K(a)$.
\end{proof}

\begin{cor}
If $G$ is an obs* graph and $K=G[G^\circ]$, then $K^\circ=K$.
\end{cor}

\begin{proof}
By Theorem~\ref{K_pd} and Theorem~\ref{regular_gnot}, $K$ is point determining and regular. If $K^\circ \neq K$ then there exist $x$, $a$ and $u$ in $K$ such that  $N_K(a)=N_K(u)-\{x\}$. Therefore, $\deg_K(a)=\deg_K(u)-1$, which contradicts Theorem~\ref{regular_gnot}.
\end{proof}

The structure of an obs* graph $G$, if all of its vertices are nucleus vertices, had been already characterized.

\begin{prop}~\cite{guzman2024full} \label{regular_v_tran}
Let $G$ be a point determining regular graph. There is a graph $H$ such that $G \in \mathrm{obs}^*(H)$ if and only if $G$ is a vertex-transitive graph.
\end{prop}

As stated in Corollary~\ref{deg_in_obs}, $G$ cannot be regular if it has any non-nucleus vertices.

The structure of Bi-sequences and belts change when some nucleus vertex $x$ is removed from a point determining graph $G$. After removing $x$, only triples containing $x$ are destroyed, and triples which do not contain $x$ are again triples in $G-x$. But some new triples may appear. Therefore we have the following proposition.

\begin{prop} \label{H_nucleuses}
Let $G$ be a point determining graph and $x \in G^\circ$. A nucleus vertex in $G-x$ is either a nucleus vertex of $G$ or a successor of $x$ in a triple component $C$.
\end{prop}

\begin{lem}~\cite{feder2008realizations} \label{small_c}
Suppose $G$ is a realization of $H$. Suppose further that $x$ is the nucleus vertex of a small belt $C$ of $G$. Then all the triples of $H$ are also triples of $G$.
\end{lem}

\begin{assum}
{\normalfont
From Lemma~\ref{merge} untill Lemma~\ref{new_on_y} we employ the following assumptions. The graph $G$ is a realization of $H$, i.e. $G \in \mathrm{obs}^*(H)$; $x_0$ is a nucleus vertex of $G$ and $\deg_G(x_0)=k$, we have $H \cong G-x_0$, so we consider $H$ as the subgraph $G-x_0$ of $G$. Let $C_0$ be a triple component of $G$ that has $x_0$ as an endpoint. Let $(y_0,\{z_0,x_0\})$ be a triple of $C_0$. First suppose that $C_0$ is not a small belt. Then some triple of $C_0$ remains in $C-x_0$. The triples of $C_0-x_0$ are contained in some triple component $C'_0$ of $G-x_0$. If $C_0$ is a small belt, according to Lemma~\ref{small_c} there exists no triple component $C'_0$. Note that in this case, $y_0$ and $z_0$ are true twins.}
\end{assum}

We call any triple of $H$ which is not a triple of $G$, a new triple of $H$.

\begin{lem}~\cite{feder2008realizations} \label{merge}
Let $T$ be any new triple of $H$. Then

\begin{enumerate}
\item $T$ contains one of $y_0,z_0$.
\item \label{circulation} If $T$ is in $C'_0$, then $C'_0$ is a type-three triple component (belt), and $T=(u_{r+1},\{v_r,u_1\})$ with $y_0=u_1,z_0=v_1$.
\item If $T$ is not in $C'_0$, then either
    \begin{enumerate}
    \item \label{seq_on_y} $T=(t,\{w,y_0\})$ with $y_0$ a nucleus vertex in $H$ and $t$ not a vertex of $C'_0$; the vertex $w$ is either not in $C'_0$ or it is the other endpoint of $C'_0$; or
    \item \label{seq_with_z} $T=(z_0,\{t',w'\})$ where $z_0$ is a non-nucleus vertex in $H$, and $C'_0$ is a small type-two triple component (small sequence); neither $t'$ nor $w'$ is in $C'_0$.
    \end{enumerate}
\end{enumerate}
\end{lem}

\begin{lem} \label{nuc_y}
The vertex $y_0$ is a nucleus vertex of $H$. 
\end{lem}

\begin{proof}
Suppose, to the contrary that $y_0$ is a non-nucleus vertex in $H$. Therefore $y_0$ is the middle vertex of some triple $T_{y_0}=(y_0,\{p,q\})$ in $H$. The triple $T_{y_0}$ is not in $G$ because by Corollary~\ref{no_intersection_in_obs} no vertex can be the middle vertex of two triples in an obs* graph $G$. Thus $T_{y_0}$ is a new triple of $H$ and therefore by Lemma~\ref{small_c}, $C_0$ is not a small belt. By the possible shapes of triple components, $T_{y_0}$ is in $C'_0$ but by Lemma~\ref{merge} if there is a new triple of $H$ in $C'_0$, then $y_0$ is a nucleus vertex.
\end{proof}

\begin{lem} \label{cprime_is_c_minus_x}
The triple component $C'_0$ has no vertices more than $C_0-x_0$ (though there may appear some new triples in $C'_0$).
\end{lem}

\begin{proof}
Suppose, to the contrary, that $C'_0$ has a vertex off $C_0-x_0$. If $C_0$ is a small belt, then by Lemma~\ref{nuc_y}, $y_0$ and $z_0$ are nucleus vertices. Then $C'_0$ must be a Bi-sequence or an odd sequence with endpoints $y_0$ and $z_0$. Let $T_{C'_0}$ be the triple of $C'_0$ that contains $y_0$. Thus, $T_{C'_0}$ is not a triple in $G$, because it is not in $C_0$ and by Corollary~\ref{no_intersection_in_obs} the non-nucleus vertex $y_0$ cannot lie in two triple components in $G$. But this contradicts Lemma~\ref{small_c}. Hence $C_0$ is not a small belt.

If $C_0$ is a Bi-sequence with one endpoint $x_0$, let $a$ be the other endpoint of $C_0$. If $C_0$ is a belt that is not small, set $a$ to be the successor of $x_0$ in the opposite direction of $y_0$. The vertices that are in $C'_0$ but not in $C_0-x_0$ must be added to the ends of $C_0-x_0$ to form a triple component in $H$. By Lemma~\ref{nuc_y}, $y_0$ is a nucleus vertex of $H$ and therefore an endpoint of $C'_0$; so these new vertices are added to the $a$-end of $C_0-x_0$. But this cannot happen if $C_0$ is a belt, otherwise $a$ is also a nucleus vertex of $G-x_0$. Thus $C_0$ is a Bi-sequence and $a$ is the other endpoint of it.

So, $a$ and some vertex $b$ out of $C_0-x_0$ are in a triple $T_0=(a,\{b,c\})$ of $C'_0$. Since $T_0$ contains $a$, it is not a triple of $C_0$. Since $a$ is a nucleus vertex of $G$, it cannot be the middle vertex of a triple in $G$, so $T_0$ is not a triple in $G$.
But by Lemma~\ref{merge}, a new triple of $H$ which is in $C'_0$ contains at least one of $y_0$ or $z_0$. But the order of $C_0-x_0$ is at least $3$, because the order of $C_0$ is at least $4$. So, $T_0$ cannot contain any of $y_0$ and $z_0$, which is a contradiction.
\end{proof}

So, the triple component $C'_0$ has the same vertices as $C_0-x_0$. However, if $C_0$ was an A-sequence $C_0=(u_1,v_1,u_2,v_2, \ldots, u_r, v_r)$ (with $u_1=x_0$, $v_1=y_0$ and $u_2=z_0$) it may now become the belt $C'=(u'_1,v'_1, \ldots, u'_r, u'_1)$ (with $u'_i=v_i$, $v'_{i-1}=u_i$ for all $2 \leq i \leq r$) in $H$, after gaining one additional triple $T=(u'_{r+1},\{v'_r,u'_1\})$. Part~\ref{seq_on_y} of Lemma~\ref{merge} says that there could be other triple components on $y_0$, which is the new nucleus vertex appearing in $G-x_0$.

In Lemma 2.8 of~\cite{feder2008realizations} (Lemma~\ref{merge} in the present paper) the following fact is not mentioned.

\begin{lem} \label{no_other_end_of_c}
In Lemma~\ref{merge} part~\ref{seq_on_y}, the vertex $w$ is not the other endpoint of $C'_0$ (one endpoint is $y_0$).
\end{lem}

\begin{proof}
Let $T=(t,\{w,y_0\})$ be a new triple in $H$. Then $N_G(y_0) \triangle N_G(w)=\{x_0,t\}$.
Suppose $w$ is the other endpoint of $C'_0$. If $C_0$ is a Bi-sequence, then by Lemma~\ref{cprime_is_c_minus_x}, $w$ is the other endpoint of $C_0$; thus $y_0$ and $w$ are either both adjacent to $x_0$, or both non-adjacent to $x_0$. If $C_0$ is a belt, then $y_0$ and $w$ are the successors of $x_0$ in opposite directions. Therefore both $y_0$ and $w$ are non-adjacent to $x_0$. In any case $x_0 \notin N_G(y_0) \triangle N_G(w)$, which is a contradiction.
\end{proof}

\begin{lem} \label{t_nucleus}
Let $T=(t,\{w,y_0\})$ be the new triple of $H$ introduced in Lemma~\ref{merge} part~\ref{seq_on_y}. Then $t$ is a nucleus vertex of $G$.
\end{lem}

\begin{proof}
If $t$ is non-nucleus in $G$, then it is the middle vertex of some triple in a triple component $D$ of $G$.
The vertex $t$ is not a successor of $x_0$ in $D$ otherwise it will be a nucleus vertex of $H$ by Lemma~\ref{nuc_y}. Hence $D$ cannot be a small belt on $x_0$. The triples of $D$ that do not contain $x_0$ are also triples in $H$. Suppose $D'$ is the triple component of $H$ containing $D-x_0$ ($x_0$ may be outside of $D$). Since $t$ is not a successor of $x_0$ in $D$, it is the middle vertex of a triple in $D'$. Thus $t$ is the middle vertex of two triples in $H$; one is in $D'$ and the other is $T$. Let $F'$ be the triple component of $H$ that contains $T$. Therefore, by Lemma~\ref{no_c_intersection}, $D'$ and $F'$ are small sequences. Consequently, $F'$ is the small sequence $(y_0,t,w)$, $w$ is a nucleus vertex in $H$ and $N_G(w) \triangle N_G(y_0)=\{x_0,t\}$.

Now suppose that $y_0$ is adjacent to $x_0$, then $C_0$ is an A-sequence and $\deg_G(y_0)<k$ (Observation~\ref{deg_seq}). Moreover, $w$ is non-adjacent to $x_0$. If $y_0$ is adjacent to $t$ we have $\deg_G(w) < \deg_G(y_0)$. If $y_0$ is non-adjacent to $t$ we have $\deg_G(w)=\deg_G(y_0)$.
If $w$ is a nucleus vertex of $G$, then $\deg_G(w)=k$. If $w$ is a non-nucleus vertex of $G$, then it is a successor of $x_0$ in some triple component of $G$ by Proposition~\ref{H_nucleuses}. Also $w$ and $x_0$ are non-adjacent. It follows that the triple component of $G$ containing $w$ is an N-sequence or a belt, thus $\deg_G(w) > k$. In either case we see a contradiction.

Finally suppose that $y_0$ is non-adjacent to $x_0$. Then $C_0$ is an N-sequence or a belt and $\deg_G(y_0) > k$. Moreover $w$ is adjacent to $x_0$. A similar argument results in $\deg_G(w) \geq \deg_G(y_0)$. Again if $w$ is a nucleus vertex of $G$, we have $\deg_G(w)=k$. If $w$ is a non-nucleus vertex of $G$, then it is a successor of $x_0$ in some triple component of $G$ by Proposition~\ref{H_nucleuses}. Also $w$ and $x_0$ are adjacent. It follows that the triple component of $G$ containing $w$ is an A-sequence, thus $\deg_G(w) < k$. In either case we see a contradiction.
\end{proof}

Let $G$ be a point determining graph and $x$ be a nucleus vertex of $G$. Suppose $F=(a_1,a_2,\ldots,a_{s-1},a_s)$ is an odd sequence or a Bi-sequence, that is not on $x$, in $G$. If there is a vertex $b$ such that $(a_s,\{a_{s-1},b\})$ is a triple in $G-x$, we say that $b$ is appended to $F$.

\begin{lem} \label{new_on_y}
 Suppose $F'$ is the triple component introduced in the proof of Lemma~\ref{t_nucleus}. Then there exists a Bi-sequence $F$ in $G$ such that
 \begin{enumerate}
 \item $x_0$ is not an endpoint of $F$.
 \item $y_0 \notin F$.
 \item $F \cup \{y_0\} \subseteq F'$.
 \end{enumerate}
\end{lem}

\begin{proof}
Let $T=(t,\{w,y_0\})$ be the new triple of $H$ introduced in Lemma~\ref{merge} part~\ref{seq_on_y}. By Lemma~\ref{t_nucleus}, $t$ is a nucleus vertex in $G$. By~\cite[page 1643]{feder2008realizations}, since $G$ contains triples (at least the triple $(y_0,\{z_0,x_0\})$), every nucleus vertex of $G$ is in some triple component. So, $t$ is contained in some triple component $F$ of $G$. We will show that $F$ has the desired properties (1), (2) and (3).

First we show that $w$ is a successor of $t$ in $F$. We know that $t$ is non-nucleus in $H$. Suppose $w$ is not a successor of $t$ in $F$. The triple $T$ and $F$ are not contained in the same triple component of $H$.
Therefore, $t$ is the intersection of two triple components. So, by Lemma~\ref{no_c_intersection}, $t$ is the middle vertex of two small sequences. If $F$ is a belt, removing $x_0$ does not alter $F$, because there is no way a belt can be extended to form another triple component. Thus $t$ is a nucleus vertex in $H$, which is a contradiction.
If $F$ is a Bi-sequence, then $t$ is at least the third vertex in the corresponding sequence of $F$ (even if $x_0$ is the other endpoint of $F$). Hence $t$ is not the middle vertex of a small sequence. So, $w$ is a successor of $t$ in $F$.

So, $y_0$ is appended to $F$ and we have a larger triple component than $F$ in $H$. Moreover, $F$ must be a Bi-sequence because, as we said, belts of $G$ could not extend in $H$.
By Lemma~\ref{cprime_is_c_minus_x}, $x_0$ is not an endpoint of $F$. By Lemma~\ref{nuc_y} $y_0$ is a nucleus vertex of $H$ and, therefore, an endpoint of $F'$ (Possibly $F'$ could be extended from the other endpoint of $F$).
\end{proof}

Now we prove that $G[G^\circ]$ is vertex-transitive.


\begin{thm} \label{the_big_thm}
If $G$ is an obs* graph, then $K=G[G^\circ]$ is a vertex-transitive graph.
\end{thm}

\begin{proof}
Let $G$ be an obs* graph. If $G$ has no non-nucleus vertex, the result is obvious by Proposition~\ref{regular_v_tran}. So, suppose that $G$ has some non-nucleus vertex, and therefore $G$ contains triple components. Let $S$ be a triple component of the shortest order $m$ among the non-small triple components ($m>3$). The case that there is no such triple component (whenever all triple components are small) is explained at the end. Suppose $x$ is an endpoint of $S$ and $T=(y,\{x,z\})$ is a triple of $S$.

Set $H=G-x$. First we look at the vertex degrees in $H$. Let $k$ be the degree of any vertex of $G^\circ$ in $G$. If $u$ is a nucleus vertex in $G$, then $\deg_G(u)=k$. So, $\deg_H(u)=k$ if $u$ is not adjacent to $x$; and $\deg_H(u)=k-1$ if $u$ is a neighbor of $x$.

By Proposition~\ref{H_nucleuses}, after removing $x$ no vertex of $G$ becomes a nucleus vertex of $H$ except those that are a successor of $x$ in the corresponding sequence of some triple component $L$; suppose $y'$ is such a vertex. Then, if $L$ is an A-sequence, then $\deg_G(y') \leq k-1$ and $y'$ is adjacent to $x$. Therefore $\deg_H(y') \leq k-2$. If $L$ is either an N-sequence or a belt, then $\deg_G(y') \geq k+1$ and $y'$ is not adjacent to $x$. Thus $\deg_H(y') \geq k+1$.

Let $S'$ be the triple component of $G-x$ containing $S-x$. By Lemmas~\ref{cprime_is_c_minus_x} and ~\ref{merge} the vertex set of $S'$ and $S-x$ are the same. As a triple component, $S'$ is either just $S-x$ or a belt containing $S-x$ together with an additional triple.
The order of $S'$ is $m-1$. Moreover if both $S$ and $S'$ are not belts, the other nucleus vertex of $S'$ is a nucleus vertex of $G$ and has degree $k$ or $k-1$ in $H$. If $S$ is not a belt and $S'$ is a belt, $S'$ has one nucleus vertex. If $S$ is a belt, then $S'$ is not a belt and the two nucleus vertices of $S'$ have the same degree greater than $k$.


Now let $a$ be a nucleus vertex of $G$. By Theorem~\ref{obs_is_realization} there exists an isomorphism $\varphi:G-x \to G-a$. Let $D'=\varphi(S')$ and $b=\varphi(y)$.
We know that $y$, as an endpoint of $S'$, has degree not equal to $k$ or $k-1$ in $H$. Thus $\deg_{G-a}(b) \neq k, k-1$.
Hence $b$ is a successor of $a$ in some triple component of $G$. In the sequel we show that this component is in fact isomorphic to $S$ and $\varphi:G-x \to G-a$ can be extended to an automorphism of $G$ with $\varphi(x)=a$.

Since $b \in (G-a)^\circ - G^\circ$, we have two types of triple components on $b$ in $H$. Consequently $D'$ is either a triple component of $G-a$ containing $D-a$ where $D$ is a triple component of $G$ on $a$, which is uniquely determined by $D'$. Or $D'$ is a triple component that is obtained from appending $b$ to some Bi-sequence $F$ of $G$, that is not on $a$, by Lemma~\ref{new_on_y}. But this latter triple component has order at least $m+1$, which is impossible by the choice of $D'$, because $D'=\varphi(S')$ and $|\varphi(S')|=|S'|=|S|-1<m$. So, the former case occurs and $b$ is a successor of $a$ in $D$. Let $\varphi(z)=c$.

We want to show that there exists an automorphism of $G$ that maps $x$ to $a$ (this desired automorphism may not be an extension of $\varphi$). The establishment of this automorphism depends on the type of $S$.

\begin{enumerate}

\item If $S$ is an N-sequence, then $S'$ is just $S-x$, and therefore an odd sequence in $H$. So, one of its endpoints has degree $\deg_H(y)$ and the other endpoint has degree $k$. Since $S'\cong D'$, $D$ is an N-sequence on $a$. It follows that the choices of $b$ and $c$ are unique. Thus, $x$ is actually the vertex that is almost false twin to $z$; more percisely $N_G(x)=N_G(z)-\{y\}$. Similarly, $N_G(a)=N_G(c)-\{b\}$. Therefore, since $\varphi(z)=c$, $\varphi$ can be extended to an automorphism $\tilde{\varphi}$ of $G$ that maps $x$ to $a$.

\item If $S$ is an A-sequence, $S'$ is either $S-x$ or a belt, as stated in Lemma~\ref{merge}. If $S'=S-x$, then again as in (1), the choices of $b$ and $c$ are unique; and we have $N_G(z)=N_G(x)-\{y\}$ and $N_G(c)=N_G(a)-\{b\}$. Therefore, since $\varphi(z)=c$, $\varphi$ can be extended to an automorphism $\tilde{\varphi}$ of $G$ that maps $x$ to $a$.

If $S'$ is a belt, then the choice of $b$ is unique, but the choice of $c$ is not unique. There are actually two options $c'$ and $c''$ for $c$; $c'$ is the vertex that makes the triple $(b,\{a,c'\})$ in $G$, $c''$ is the other endpoint of $D$. Since $N_G(z)=N_G(x)-\{y\}$ and $S' \cong D'$ is a belt, $D$ must be an A-sequence on $a$, because other triple components become a Bi-sequence or an odd sequence after removing one endpoint by Lemma~\ref{merge}. If $\varphi$ maps $z$ to $c'$, we can extend $\varphi$ as when $S'$ is not a belt. If $\varphi$ maps $z$ to $c''$, then we first define the following isomorphism $\psi:G-x \to G-a$ and extend it to an automorphism $\tilde{\psi}:G \to G$ that maps $x$ to $a$.

To get $\psi$, first we introduce an automorphism $f:G-a \to G-a$ that only moves the vertices of $D'$ such that maps $c''$ to $c'$. If $D'=(u_1,v_1,u_2,v_2, \ldots, u_r, v_r, u_{r+1},u_1)$, with $v_1=c'$ and $u_{r+1}=c''$, we define

\[
f(x)=
\begin{cases}
v_{r-i+2} & \text{if} \quad x=u_i,\\
u_i & \text{if} \quad x=v_{r-i+2},\\
x & \text{otherwise.}
\end{cases}
\]

In fact, $f$ is the reflection of the belt $D'$ with respect to its middle diameter $\left(u_{\lfloor\frac{r+1}{2}\rfloor},v_{\lfloor\frac{r}{2}\rfloor+1}\right)$. Then by the definition of a belt, $f$ is an automorphism of $G-a$. Now $\psi=f \circ \varphi:G-x \to G-a$ is an isomorphism which can be extended to the automorphism $\tilde{\psi}:G \to G$ by defining $\tilde{\psi}(x)=a$.

\item If $S$ is a belt, then $S'=S-x$ and since $S$ is not small, we can do the same as in (1). Note that in this case, there are actually two options $(b',c')$ and $(b'',c'')$ for the pair $(b,c)$, which make the triples $(b',\{a,c'\})$ and $(b'',\{a,c''\})$ in the opposite directions of $S$. But it does not matter which one we choose
because $S'$ is actually a triple component of $H$ with two endpoints of degree not equal to $k$ or $k-1$. The only way that we can have such a triple component ($D' \cong S'$) in $G-a$ is to have a belt $D$ on $a$. So, $D$ has to be a belt and we have $N_G(a)=N_G(c')-\{b'\}$ and also $N_G(a)=N_G(c'')-\{b''\}$.

\end{enumerate}

Now suppose that there is no triple component of order more than $3$. In this case all triple components are small belts. Let $S$ be a triple component on a nucleus vertex $x$ of $G$ (which is a small belt). Then we do not have any triple component $S'$ of $G-x$ that contains $S-x$ as we saw in Lemma~\ref{cprime_is_c_minus_x}. Because $S-x$ has two vertices and no vertex can be appended to it. But we have two nucleus vertices $y$ and $z$ of degree $k+1$ that are true twins in $G-x$. Let $a$ be any nucleus vertex of $G$. There exists an isomorphism $\varphi:G-x \to G-a$. Thus we have two nucleus vertices $\varphi(y)=b$ and $\varphi(z)=c$ of degree $k+1$ in $G-a$, therefore both lie on a small belt on $a$ in $G$, say $D=(a,b,c,a)$.
But then we have $N(x)=N(y)-\{z\}$ and $N(a)=N(b)-\{c\}$, and so we can, again, extend $\varphi$ to $\tilde{\varphi}$ on $G$ by mapping $x$ to $a$. In this case the choice of $b$ and $c$ is not unique, but it is not important which one we choose, because $b$ and $c$ are true twins and we also have $N(a)=N(b)-c$.

In both cases, $|S|>3$ and $|S|=3$ with $x$ an endpoint of $S$ we provided
an automorphism $\tilde{\varphi}:G \to G$ such that $\tilde{\varphi}(x)=a$, for every $a \in G^\circ$. Now for each two nucleus vertices $a$ and $a'$, there exist automorphisms $\varphi_a$ and $\varphi_{a'}$ of $G$ such that $\varphi_a(x)=a$ and $\varphi_{a'}(x)=a'$. Then $\varphi_{a'} \circ \varphi_a^{-1}$ is an automorphism of $G$ that maps $a$ to $a'$. Furthermore, every automorphism of $G$ maps $G^\circ$ to $G^\circ$. Therefore $G^\circ$ is vertex-transitive and the proof is complete.
\end{proof}

In fact, in Theorem~\ref{the_big_thm} we have proved the following corollary.

\begin{cor} \label{aut_G}
Let $G$ be an obs* graph and $K=G[G^\circ]$. For every two vertices $x,a \in K$ there exists $\psi \in \mathrm{Aut}(G)$ such that $\psi(x)=a$. Also the restriction of automorphisms of $G$ to $G^\circ$ makes a subgroup of $\mathrm{Aut}(K)$.
\end{cor}


\section{Characterization of obs* Graphs} \label{main_result}

In this section we completely characterize obs* graphs.

Let $K$ be a vertex transitive graph. A \textit{transitive subgroup} of $\mathrm{Aut}(K)$ is a subgroup $\Gamma \leq \mathrm{Aut}(K)$ such that for every $u,v \in K$ there exists $\gamma \in \Gamma$ that $\gamma(u)=v$. We say that two edges $xy$ and $uv$ belong to the same edge orbit of $G$ with respect to $\Gamma$, if there exists an element $\gamma \in \Gamma$ such that $\gamma(x)=u$ and $\gamma(y)=v$. We call two non-adjacent vertices of $G$ a non-edge of $G$. The non-edge orbits of $G$ with respect to $\Gamma$ are defined similarly (they are the edge orbits of the complement of $G$).

\begin{ex} \label{patric_orbits}
Let $K=C_6$ with vertices $u_1,\ldots,u_6$. The automorphism group of $C_6$ is (isomorphic to) the dihedral group $D_{12}$. Let $\mathrm{Aut}(C_6)=D_{12}=\{e,a,a^2,a^3,a^4,a^5,b,ab,a^2b,a^3b,a^4b,a^5b\}$ where $a$ denotes the $\frac{2\pi}{6}$ counterclockwise rotation and $b$ denotes the reflection $u_1 \longleftrightarrow u_6$, $u_2\longleftrightarrow u_5$ and $u_3 \longleftrightarrow u_4$. The subgroup $\Gamma=\{e,a^2,a^4,b,a^2b,a^4b\} \leq D_{12}$ is a transitive subgroup of $\mathrm{Aut}(C_6)$. The edge orbits of $C_6$ with respect to $\Gamma$ are

\[
\mathcal{O}_1=\{u_6u_1,u_2u_3,u_4u_5\},\mathcal{O}_2=\{u_1u_2,u_3u_4,u_5u_6\}.
\]

The non-edge orbits of $C_6$ with respect to $\Gamma$ are

\[
\mathcal{O}'_1=\{u_1u_4,u_2u_5,u_3u_6\},\mathcal{O}'_2=\{u_1u_5,u_5u_3,u_3u_1\},\mathcal{O}'_3=\{u_2u_4,u_4u_6,u_6u_2\}. 
\]

\end{ex}

Let $K$ be a graph and $x$ and $y$ be two vertices of $K$. Suppose $C$ is a belt disjoint from $K$. We say we put $C$ on the vertex $x$, if we identify the endpoint of $C$ with the vertex $x$ (this operation is also known as a vertex-sum in the literature). Let $\mathcal{A}$ be an A-sequence disjoint from $K$ and $xy$ be an edge of $K$. We say we put $\mathcal{A}$ on $xy$ if we identify the vertices $x$ and $y$ with the endpoints of $\mathcal{A}$ (this operation is also known as an edge-sum in the literature). Note that it is does not matter which of $x$ or $y$ are identified with which endpoint of $\mathcal{A}$, because an A-sequence is symmetric with respect to its endpoints. Similarly, if $xy$ is a non-edge of $K$ and $\mathcal{N}$ is an N-sequence disjoint from $K$, we define putting $\mathcal{N}$ on $xy$ if we identify the vertices $x$ and $y$ with the endpoints of $\mathcal{N}$. Again note that it is does not matter which of $x$ or $y$ are identified with which endpoint of $\mathcal{N}$, because an N-sequence is symmetric with respect to its endpoints. We mention that a Bi-sequence that is put on $K$ may be extended to a larger triple component in the new graph.

The following construction explains how an obs* graph is achieved.

Take a point determining vertex-transitive graph $K$. Choose any transitive subgroup $\Gamma \leq \mathrm{Aut}(K)$. Choose a set of belts (possibly empty) of arbitrary orders and put each of them on each vertex of $K$. Consider the orbits of edges of $K$ with respect to $\Gamma$. For each orbit of edges, choose a collection of A-sequences (possibly empty) of arbitrary orders and put each of them on each edge of this orbit. Consider the orbits of non-edges of $K$ with respect to $\Gamma$. For each orbit of non-edges, choose a collection of N-sequences (possibly empty) of arbitrary orders and put each of them on each non-edge of this orbit. Whenever we put the same A-sequence on two different edges of $K$, they are named mates. Mate N-sequences and mate belts are defined similarly.

Other adjacencies of this graph are defined with the same way as Remark~\ref{other_adj}.
We do not have any odd sequences in this graph, so every vertex is either itself a nucleus vertex or is almost false twin with a nucleus vertex. In a Bi-sequence $(u_1,v_1,u_2,v_2,\ldots,u_r,v_r)$, the vertex $u_i$, $i>1$, is almost false twin with $u_1$, and the vertex $v_j$, $j<r$ is almost false twin with $v_r$ (the endpoints $u_1$ and $v_r$ are nucleus vertices). In a belt $(u_1,v_1,\ldots,u_r,v_r,u_{r+1},u_1)$, every non-nucleus vertex $z \neq u_1$ is almost false twin with the endpoint $u_1$ (which is a nucleus vertex).
Let $u$ and $v$ be two vertices that are not in the same triple component of $G$. If $u$ is nucleus, set $u_0=u$ and if $u$ is non-nucleus, set $u_0$ to be the nucleus vertex that is almost false twin with $u$. We do the same for $v$ to get $v_0$. Now $u_0v_0$ is an edge, if and only if $uv$ is an edge.
Also note that any two Bi-sequence and/or belts can only intersect in their endpoints, by the definition.
We temporarily call this graph a good graph and denote it by $G(K,\Gamma)$.

We claim that a good graph $G(K,\Gamma)$ is an obs* graph and is, in fact, in $\mathrm{obs}^*(G-x)$ for each $x \in K$. Conversely, we show that every obs* graph is a good graph.

\begin{ex} \label{crab_example}
Consider the graph $K=2K_2$ with vertices $u_1,u_2,u_3,u_4$ and edges $u_1u_2,u_3u_4$, which is a point determining vertex-transitive graph. The group of isomorphisms of $2K_2$ is (isomorphic to) the dihedral group $D_8$. Let $\mathrm{Aut}(2K_2)=D_8=\{e,a,a^2,a^3,b,ab,a^2b,a^3b\}$ where $a$ denotes the automorphism $(u_1u_3u_2u_4)$ and $b$ denotes the automorphism $(u_1u_2)$. Then $\Gamma=\{e,a^2,ab,a^3b\}=\{1,(u_1u_2)(u_3u_4),(u_1u_3)(u_2u_4),(u_1u_4)(u_2u_3)\} \leq D_8$ is a transitive subgroup of $\mathrm{Aut}(2K_2)$. We have only one edge orbit of $2K_2$ with respect to $\Gamma$,

\[
\mathcal{O}_1=\{u_1u_2,u_3u_4\}.
\]

The non-edge orbits of $2K_2$ with respect to $\Gamma$ are

\[
\mathcal{O}'_1=\{u_2u_4,u_1u_3\},\mathcal{O}'_2=\{u_1u_4,u_2u_3\}. 
\]

We put two A-sequences of order $4$ and one A-sequence of order $6$ on each edge of $\mathcal{O}_1$. We put an N-sequence of order $4$ on each non-edge of $\mathcal{O}'_1$. We put a small belt on each vertex of $2K_2$. The resulted graph is a good graph, which we call it the Crab graph, Figure~\ref{crab}.


In this Figure, first we have drawn the graph $K$ with black vertices. The Bi-sequences are shown as handle paths and the belts as polygons. We only mention the adjacencies and non-adjacencies of the consecutive vertices that are in one Bi-sequence or belt, to avoid cluttering the figure. Note that every A-sequence starts and ends with edges and every N-sequence starts and ends with non-edges. Other adjacencies between different Bi-sequences and belts that we omit their drawings can be determined by the structure of triples in Bi-sequences and belts. For example, consider three vertices $x,y,z$ in the Crab graph. The vertex $x$ is almost false twin with $u_4$ and the vertex $y$ is almost false twin with $u_3$ and the vertex $z$ is also almost false twin with $u_3$. The vertices $u_3$ and $u_4$ are adjacent, thus $x$ and $y$ are adjacent. Moreover, $x$ and $z$ are adjacent. Since $u_3$ is not adjacent to itself (we do not have loops), $y$ and $z$ are non-adjacent.


\begin{figure}[h]

\centering

\begin{tikzpicture}[
    x=2cm,
    y=2cm,
    vertex/.style={draw,circle,fill=white,inner sep=0pt,minimum size=6pt,thick},
    edge/.style={draw,thick},
]

\begin{scope}[rotate=90]


\node[vertex,fill=black] (u1) at (-1.2, 1.2) {};
\node[vertex,fill=black] (u2) at ( 1.2, 1.2) {};
\node[vertex,fill=black] (u3) at ( 1.2,-1.2) {};
\node[vertex,fill=black] (u4) at (-1.2,-1.2) {};

\draw[edge] (u1)--(u2);
\draw[edge] (u3)--(u4);


\foreach \i [evaluate=\i as \nexti using {int(mod(\i,4)+1)},
             evaluate=\i as \previ using {int(mod(\i+2,4)+1)}] in {1,...,4} {

    \coordinate (dir\i) at ($ (u\previ) + (u\nexti) - (u\i) $);

    \coordinate (midA\i) at ($(u\i)!-0.16!(dir\i)$);

    \node[vertex] (p\i) at ($(midA\i)!0.8!90:(u\i)$) {};
    \node[vertex] (q\i) at ($(midA\i)!0.8!-90:(u\i)$) {};

    \coordinate (midB\i) at ($(u\i)!-0.32!(dir\i)$);

    \node[vertex] (r\i) at ($(midB\i)!0.4!90:(u\i)$) {};
    \node[vertex] (s\i) at ($(midB\i)!0.4!-90:(u\i)$) {};

    \draw[edel] (u\i)--(p\i);
    \draw[edge] (p\i)--(r\i);
    \draw[edel] (r\i)--(s\i);
    \draw[edge] (s\i)--(q\i);
    \draw[edel] (q\i)--(u\i);
}



\foreach \A/\B in {u1/u2,u3/u4} {

    \node[vertex] (a\A\B) at ($(\A)!0.30!(\B)!0.30!90:(\B)$) {};
    \node[vertex] (b\A\B) at ($(\A)!0.70!(\B)!0.75!90:(\B)$) {};

    \draw[edge] (\A)--(a\A\B);
    \draw[edel] (a\A\B)--(b\A\B);
    \draw[edge] (b\A\B)--(\B);

    \node[vertex] (c\A\B) at ($(\A)!0.35!(\B)!0.18!90:(\B)$) {};
    \node[vertex] (d\A\B) at ($(\A)!0.65!(\B)!0.32!90:(\B)$) {};

    \draw[edge] (\A)--(c\A\B);
    \draw[edel] (c\A\B)--(d\A\B);
    \draw[edge] (d\A\B)--(\B);

    \node[vertex] (e\A\B) at ($(\A)!0.25!(\B)!0.45!90:(\B)$) {};
    \node[vertex] (f\A\B) at ($(\A)!0.75!(\B)!1.4!90:(\B)$) {};
    \node[vertex] (g\A\B) at ($(\A)!0.5!(e\A\B)$) {};
    \node[vertex] (h\A\B) at ($(f\A\B)!0.5!(\B)$) {};

    \draw[edge] (\A)--(g\A\B);
    \draw[edel] (g\A\B)--(e\A\B);
    \draw[edge] (e\A\B)--(f\A\B);
    \draw[edel] (f\A\B)--(h\A\B);
    \draw[edge] (h\A\B)--(\B);

}


\foreach \A/\B in {u2/u3,u4/u1} {


    \node[vertex] (a\A\B) at ($(\A)!0.30!(\B)!0.30!90:(\B)$) {};
    \node[vertex] (b\A\B) at ($(\A)!0.70!(\B)!0.75!90:(\B)$) {};

    \draw[edel] (\A)--(a\A\B);
    \draw[edge] (a\A\B)--(b\A\B);
    \draw[edel] (b\A\B)--(\B);

}


\node[label=right:$u_1$] at (u1) {};
\node[label=right:$u_2$] at (u2) {};
\node[label=left:$u_4$] at (u3) {};
\node[label=left:$u_3$] at (u4) {};
\node[label=above:$x$] at (au2u3) {};
\node[label=right:$y$] at (fu3u4) {};
\node[label=left:$z$] at (p4) {};

\end{scope}

\end{tikzpicture}

\caption{The Crab graph}
\label{crab}

\end{figure}
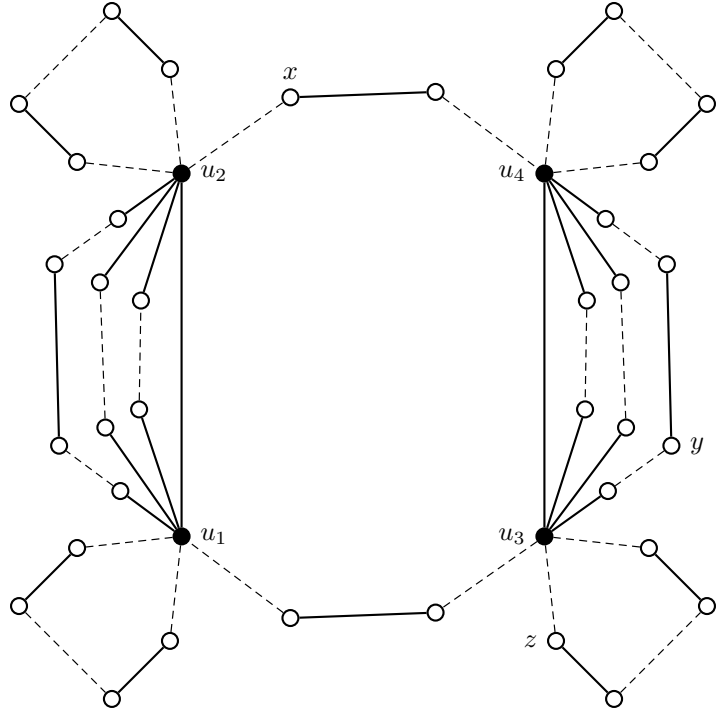


\end{ex}

\begin{lem} \label{good_autos}
Let $G=G(K,\Gamma)$ be a good graph. Then every element of $\Gamma \leq \mathrm{Aut}(K)$ can be extended to an automorphism of $G$.
\end{lem}

\begin{proof}
Let $\gamma \in \Gamma$ and $e$ be an edges of $K$. The construction of $G$ provides a one-to-one correspondence between the collection of A-sequences on $e$ and the collection of A-sequences on $\gamma(e)$. By the same token, for a non-edge $f$ there exists a one-to-one correspondence between the collection of N-sequences on $f$ and the collection of N-sequences on $\gamma(f)$. Also for any vertex $x$ of $K$, there is a one-to-one correspondence between the collection of belts on $x$ and the collection of belts on $\gamma(x)$.

Let $x$ and $a$ be two vertices of $K$, $\gamma(x)=x'$ and $\gamma(a)=a'$. To extend $\gamma$, if $x$ and $a$ are adjacent, we send an A-sequence on $xa$ to its mate on $x'a'$. Similarly we send N-sequences and belts to their mates. We name this function $\varphi$. If $u$ and $v$ are almost false twins, then $\varphi(u)$ and $\varphi(v)$ are almost false twins, too.

Now we show that $\varphi$ is an automorphism of $G$. Consider an A-sequence $\mathcal{A}=(u_1,v_1,\ldots,u_r,v_r)$ with its image $\mathcal{A}'=(u_1',v_1',\ldots,u_r',v_r')$. Let $p$ be a vertex of $G$ such that $u_ip$ is an edge. We show that $\varphi(u_i) \varphi(p)$ is an edge of $G$.

First suppose that $p \in A$. Then for some $j \geq i$, $p=v_j$. Therefore, the vertex $\varphi(p)=\varphi(v_j)=v_j'$ is adjacent to $\varphi(u_i)=u_i'$.

Second suppose that $p \in K$. Since $u_i$ and $u_1$ are almost false twins, $u_1p$ is an edge. Since $\gamma$ is an automorphism of $K$, $\varphi(u_1) \varphi(p)=\gamma(u_1) \gamma(p)$ is an edge, too. Now $\varphi(u_1)=u_1'$ and $\varphi(u_i)=u_i'$ are almost false twins, therefore $\varphi(u_i) \varphi(p)$ is an edge.

Finally suppose that $p$ is a vertex in some Bi-sequence or belt other than $\mathcal{A}$. The graph $G$ is a good graph, thus $p$ is almost false twin with some vertex $a$ of $K$. Since $u_i$ is almost false twin with $u_1$, $u_1p$ is an edge; and since $p$ is almost false twin with $a$, $u_1a$ is an edge. But both $u_1$ and $a$ are in $K$, hence $\varphi(u_1) \varphi(a)=\gamma(u_1)\gamma(a)$ is an edge. Now because $\varphi(u_1)$ and $\varphi(u_i)$ are almost false twins, $\varphi(u_i) \varphi(a)$ is an edge. Also $\varphi(a)$ and $\varphi(p)$ are almost false twins, so $\varphi(u_i) \varphi(p)$ is an edge.

Same arguments work for N-sequences and belts, mutatis mutandis. Now we have the extension $\varphi \in \mathrm{Aut}(G)$ of $\gamma \in \Gamma$.
\end{proof}

\begin{lem} \label{good_pd}
A good graph is point determining.
\end{lem}

\begin{proof}
Let $G=G(K,\Gamma)$ be a good graph. We know that $K$ is point determining, so two vertices of $K$ have different neighborhoods in $K$ and therefore cannot be false twins in $G$. We continue to prove that every other pair of vertices in $G$ are not false twins.

By Lemma~\ref{good_autos}, for each two vertices $a,x \in K$, we have an automorphism $\psi \in \mathrm{Aut}(G)$ that takes $x$ to $a$, thus $\deg_G(x)=\deg_G(a)$. Let $k$ be the degree of any vertex of $K$ in $G$. By the same argument as Observation~\ref{deg_seq}, any vertex off $K$ has degree not equal to $k$ in $G$. Middle vertices of A-sequences have degree smaller than $k$ and middle vertices of N-sequences and belts have degree greater than $k$. So, any vertex of $K$ and any vertex out of $K$ cannot be false twins.

Now consider two vertices $z$ and $c$ out of $K$. No two vertices of a Bi-sequence or a belt are false twins, therefore suppose $z$ and $c$ are in different Bi-sequences or belts. Then $z$ is almost false twin to some vertex $x$ of $K$ and $c$ is almost false twin to some vertex $a$ of $K$. If $a \neq x$, then $a$ and $x$ are not false twins in $K$, and thus there is some vertex $d$ of $K$ that is adjacent to exactly one of $x$ and $a$. It follows that $d$ is a difference in the neighborhoods of $z$ and $c$. Thus $z$ and $c$ are not false twins.

If $x=a$, then, in order to be false twins, $z$ and $c$ have the same degrees and therefore the same sequence-distance from $x$. Hence $z$ is the middle vertex of $C_1$ and $c$ is the middle vertex of $C_2$, where $C_1$ and $C_2$ are different Bi-sequences or belts. There exists a middle vertex $y$ in $C_1$ which is a difference between the neighborhoods of $x$ and $z$. But $y$ is not in $C_2$ and therefore it is not a difference between the neighborhoods of $a=x$ and $c$. Hence, $y$ is a difference between the neighborhoods of $z$ and $c$. It follows that, in any case, $z$ and $c$ are not false twins. The proof is now complete.
\end{proof}

\begin{lem} \label{G_circ_is_K}
For a good graph $G=G(K,\Gamma)$, $G^\circ=K$.
\end{lem}

\begin{proof}
If $G=K$, which means no Bi-sequences or belts are put on $K$, then $G^\circ=K^\circ=K$, since $K$ is vertex-transitive and point determining.

So, suppose there are some Bi-sequences or belts that are put on $K$. We show that $G^\circ \subseteq K$ and $K \subseteq G^\circ$. Any vertex $u$ not in $K$, is in some triple $T_u=(u,\{v,w\})$ and hence is not a nucleus vertex of $G$. Consequently, $G^\circ \subseteq K$.

Suppose $x$ is a vertex of $K$. Since $\Gamma$ is transitive, and at least we have one Bi-sequence or belt that is put on $K$, one of its mates, say $C$, is on $x$. If $x$ is not a nucleus vertex of $G$, then it is involved in some triple $T_x=(x,\{a,b\})$.

Note that $C$ cannot be a belt, because a belt cannot be extended to form a new triple component. So, $C$ is a Bi-sequence of order at least $4$. Thus $x$ is at least the forth vertex in some Bi-sequence. If a non-nucleus vertex is in two triple components, they must be small sequences by Lemma~\ref{no_c_intersection}. Therefore $x$ must be the second vertex in its corresponding sequence, which is a contradiction. Hence, if $y$ is the successor of $x$ in $C$, then $b=y$ and $T_x=(x,\{y,a\})$. Moreover, $a$ is another vertex of $K$, because of the possible shapes of triple components. By Lemma~\ref{good_autos} there is an automorphism of $G$ that takes $x$ to $a$. Therefore $\deg_G(x)=\deg_G(a)$.

If $C$ is an A-sequence, $y$ is adjacent to $x$ and $\deg_G(y)<\deg_G(x)$, by an argument similar to Observation~\ref{deg_seq}. Furthermore we have $N_G(a)=N_G(y)-\{x\}$, so $\deg_G(y)=\deg_G(a)+1$, which is a contradiction. If $C$ is an N-sequence, $y$ is non-adjacent to $x$ and $\deg_G(y)>\deg_G(x)$. Furthermore we have $N_G(y)=N_G(a)-\{x\}$, so $\deg_G(y)=\deg_G(a)-1$, which is a contradiction. Therefore $x$ is a nucleus vertex and $G^\circ=K$.
\end{proof}

Now all ingredients are ready to prove the main theorem of this paper.

\begin{thm} \label{char_obst}
A graph $G$ is an obs* graph if and only if it is a good graph $G=G(K,\Gamma)$ with $K=G[G^\circ]$ and $\Gamma=\{\varphi\big|_{K} : \; \varphi \in \mathrm{Aut}(G)\}$.
\end{thm}

\begin{proof}
First suppose that $G=G(K,\Gamma)$ is a good graph. By Lemmas~\ref{good_pd} and~\ref{G_circ_is_K}, $G$ is point determining and $G^\circ=K$. For every two nucleus vertices $x,a \in K$ we have an automorphism $\gamma$ of $K$ that sends $x$ to $a$. But by Lemma~\ref{good_autos}, an element $\gamma \in \Gamma$ can be extended to an automorphism $\psi$ of $G$. Therefore we have an isomorphism $\psi':G-x \to G-a$. Thus $G$ is a realization of the graph $H=G-x$ and is an obs* graph, by Theorem~\ref{obs_is_realization}.

For the converse, let $G$ be an obs* graph and $K=G[G^\circ]$. By Theorems~\ref{K_pd} and~\ref{the_big_thm}, $K$ is a point determining vertex-transitive graph. Let $\Gamma=\{\varphi\big|_{K} : \; \varphi \in \mathrm{Aut}(G)\} \leq \mathrm{Aut}(K)$. The subgroup $\Gamma$ is transitive, by Corollary~\ref{aut_G}.

Suppose $\mathcal{O}$ is an edge orbit of $K$ with respect to $\Gamma$ and let $e_1,e_2 \in \mathcal{O}$. Then there exists $\varphi \in \mathrm{Aut}(G)$ such that $\varphi\big|_{K} \in \Gamma$ maps $e_1$ to $e_2$. Consequently, $\varphi$ takes any A-sequence on $e_1$ to an A-sequence on $e_2$. Thus the collection of A-sequences on $e_1$ and the collection of A-sequences on $e_2$ are the same (up to the isomorphism $\varphi$). Note that if $e_1 \neq e_2$, then the collection of A-sequences on $e_1$ and the collection of A-sequences on $e_2$ do not intersect, though whenever $e_1$ and $e_2$ have a vertex $x$ in common, every A-sequence on $e_1$ intersects every A-sequence on $e_2$ in the vertex $x$. We have a similar argument for elements of non-edge orbits and N-sequences.
For each two nucleus vertices $x$ and $a$ there exists $\varphi\big|_{K} \in H$ that takes $x$ to $a$. The automorphism $\varphi \in \mathrm{Aut}(G)$ takes the belts on $x$ to the belts on $a$. So, the collection of belts on $x$ and the collection of belts on $a$ are the same (up to the isomorphism $\varphi$). Therefore $G=G(K,\Gamma)$ is a good graph.

\end{proof}

\begin{ex}
This is a continuation of Example~\ref{patric_orbits} to construct a good graph $G(K,\Gamma)$, the Patrick Star graph (Patrick Star is the starfish character from SpongeBob SquarePants), with $K=C_6$ and the transitive subgroup $\Gamma=\{e,a^2,a^4,b,a^2b,a^4b\}$ of $\mathrm{Aut}(K)$, see Figure~\ref{patrick}. We put two A-sequences of order $4$ on each edge of $\mathcal{O}_1$. We put an N-sequence of order $6$ on each non-edge of $\mathcal{O}'_1$, and nothing on other edge and non-edge orbits. We also put a small belt on each vertex of $K$. By Theorem~\ref{char_obst}, this will give us an obs* graph of order $42$. The method that we use to draw this graph is as in Example~\ref{crab_example} to avoid cluttering the figure.


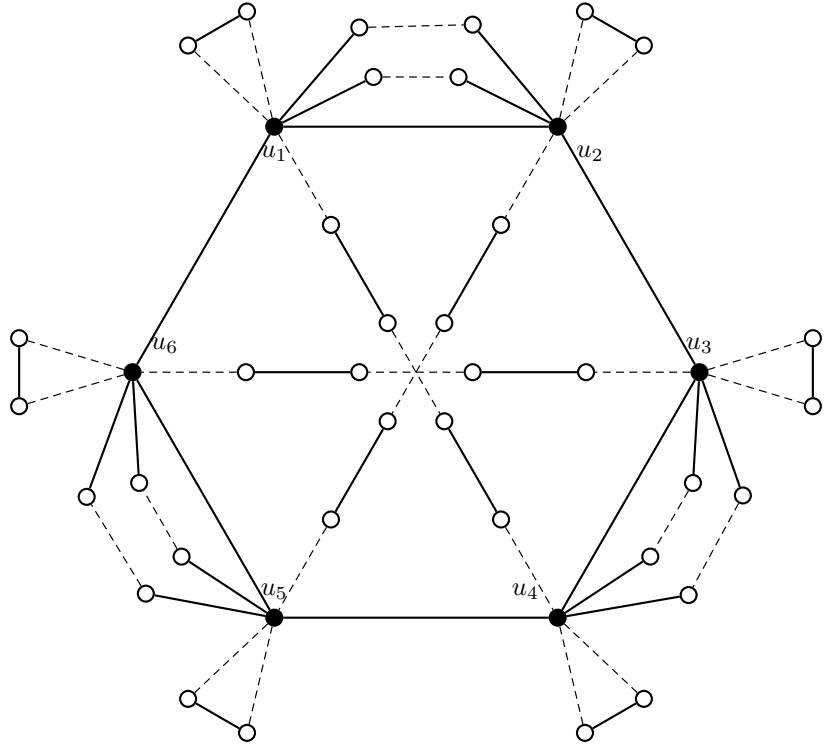
\begin{figure}[h]

\centering

\begin{tikzpicture}[
    x=2.5cm, y=2.5cm, 
    vertex/.style={draw, circle, fill=white, inner sep=0pt, minimum size=6pt, thick},
    edge/.style={draw, thick}
]

    \foreach \i in {1,...,6} {
        \node[vertex, fill=black] (u\i) at ({120 - 60*(\i-1)}:1.5) {};
    }

\foreach \i [evaluate=\i as \nexti using {int(mod(\i,6)+1)},
             evaluate=\i as \previ using {int(mod(\i+4,6)+1)}] in {1,...,6} {

    \coordinate (dir\i) at ($ (u\previ) + (u\nexti) - (u\i) $);

    \coordinate (mid\i) at ($(u\i)!-0.4!(dir\i)$);

    \node[vertex] (p\i) at ($(mid\i)!0.3!90:(u\i)$) {};
    \node[vertex] (q\i) at ($(mid\i)!0.3!-90:(u\i)$) {};

    \draw[edel] (u\i)--(p\i);
    \draw[edge] (p\i)--(q\i);
    \draw[edel] (q\i)--(u\i);

}

    \node[label={below:$u_1$}]       at (u1) {};
    \node[label={below right:$u_2$}] at (u2) {};
    \node[label={above :$u_3$}] at (u3) {};
    \node[label={above left:$u_4$}]       at (u4) {};
    \node[label={above:$u_5$}]  at (u5) {};
    \node[label={above right:$u_6$}]  at (u6) {};

    \draw[edge] (u1)--(u2)--(u3)--(u4)--(u5)--(u6)--(u1);

\node[vertex] (x14a) at ($(u1)!1/5!(u4)$) {};
\node[vertex] (x14b) at ($(u1)!2/5!(u4)$) {};
\node[vertex] (x14c) at ($(u1)!3/5!(u4)$) {};
\node[vertex] (x14d) at ($(u1)!4/5!(u4)$) {};

\draw[edel] (u1)--(x14a);
\draw[edge] (x14a)--(x14b);
\draw[edel] (x14b)--(x14c);
\draw[edge] (x14c)--(x14d);
\draw[edel] (x14d)--(u4);

\node[vertex] (x25a) at ($(u2)!1/5!(u5)$) {};
\node[vertex] (x25b) at ($(u2)!2/5!(u5)$) {};
\node[vertex] (x25c) at ($(u2)!3/5!(u5)$) {};
\node[vertex] (x25d) at ($(u2)!4/5!(u5)$) {};

\draw[edel] (u2)--(x25a);
\draw[edge] (x25a)--(x25b);
\draw[edel] (x25b)--(x25c);
\draw[edge] (x25c)--(x25d);
\draw[edel] (x25d)--(u5);

\node[vertex] (x36a) at ($(u3)!1/5!(u6)$) {};
\node[vertex] (x36b) at ($(u3)!2/5!(u6)$) {};
\node[vertex] (x36c) at ($(u3)!3/5!(u6)$) {};
\node[vertex] (x36d) at ($(u3)!4/5!(u6)$) {};

\draw[edel] (u3)--(x36a);
\draw[edge] (x36a)--(x36b);
\draw[edel] (x36b)--(x36c);
\draw[edge] (x36c)--(x36d);
\draw[edel] (x36d)--(u6);

    \foreach \i [evaluate=\i as \nexti using {int(mod(\i,6)+1)}] in {1,3,5} {
        
        \node[vertex] (a\i) at ($(u\i)!0.3!(u\nexti)!0.5!90:(u\nexti)$) {};
        \node[vertex] (b\i) at ($(u\i)!0.7!(u\nexti)!1.2!90:(u\nexti)$) {};
        \draw[edge] (u\i)--(a\i);
        \draw[edel] (a\i)--(b\i);
        \draw[edge] (b\i)--(u\nexti);

        \node[vertex] (c\i) at ($(u\i)!0.35!(u\nexti)!0.27!90:(u\nexti)$) {};
        \node[vertex] (d\i) at ($(u\i)!0.65!(u\nexti)!0.5!90:(u\nexti)$) {};
        \draw[edge] (u\i)--(c\i);
        \draw[edel] (c\i)--(d\i);
        \draw[edge] (d\i)--(u\nexti);

    }

\end{tikzpicture}

\caption{The Patrick Star graph}
\label{patrick}

\end{figure}

\end{ex}


\bibliographystyle{plain}
\bibliography{obs_bib}


\end{document}